\documentclass[a4paper]{amsart}
\usepackage{amsmath, amsthm, amssymb}
\usepackage{enumerate}

\newcommand{\cov}{\nabla}
\newcommand{\dm}{\diamondsuit}
\newcommand{\Id}{\mathrm{Id}}
\newcommand{\lie}{\mathcal{L}}
\newcommand{\lap}{\Delta}
\newcommand{\wdd}{{\wedge}}

\newcommand{\eps}{\epsilon}

\newcommand{\R}{\mathbb{R}}
\newcommand{\shuffle}{\mathrm{Sh}}
\newcommand{\vf}{\mathfrak{X}}

\DeclareMathOperator{\tr}{\mathrm{C}}
\DeclareMathOperator{\Tr}{\mathrm{tr}}
\newtheorem{theorem}{Theorem}
\newtheorem{proposition}[theorem]{Proposition}
\newtheorem{corollary}[theorem]{Corollary}

\title{Generalized Goldberg Formula}

\author[A. De Nicola]{Antonio De Nicola}
 \address{CMUC, Department of Mathematics, University of Coimbra, 3001-501 Coimbra, Portugal}
 \email{antondenicola@gmail.com}

\author[I. Yudin]{Ivan Yudin}
 \address{CMUC, Department of Mathematics, University of Coimbra, 3001-501 Coimbra, Portugal}
 \email{yudin@mat.uc.pt}


\thanks{Research partially supported by CMUC -- UID/MAT/00324/2013, funded by the Portuguese
 Government through FCT/MEC and co-funded by the European Regional Development Fund through the Partnership Agreement PT2020 (A.D.N. and I.Y.), by MICINN (Spain) grants MTM2011-15725-E, MTM2012-34478 (A.D.N.), and by the exploratory research project in the frame of Programa Investigador FCT IF/00016/2013 (I.Y.).}

\begin{document}

\begin{abstract}
In this paper we prove a useful formula for the graded commutator of the Hodge
codifferential with the left wedge multiplication by a fixed $p$-form acting on
the de Rham algebra of a Riemannian manifold. Our formula generalizes a formula
stated by Samuel I. Goldberg for the case of 1-forms. As first examples of
application we obtain new identities on locally conformally K\"ahler manifolds
and quasi-Sasakian manifolds. Moreover, we prove that under suitable conditions
a certain subalgebra of differential forms in a compact manifold is quasi-isomorphic as a CDGA to the full de Rham algebra.
\end{abstract}

\maketitle

\section{Introduction}

Since the early days of Differential Geometry the importance of formulae that relate various
differential objects on a manifold has been apparent.
Let us mention among others the Bianchi identities, Weitzenb\"ock formulae, and
Fr\"olicher-Nijenhuis calculus. It should be noted that all the above results
can be obtained by elementary, although tedious and long, computations.  Their
importance lies in the psychological and practical plane, as they permit to work with the
quantities in question without undergoing error-prone calculations, thus
forming a swiss-knife kit of a differential geometer.
In this article we prove a formula that we hope will deserve its place in the
kit.

Let $(M,g)$ be a Riemannian manifold. As usual, $\Omega^*(M)$ denotes the de Rham algebra of differential forms on $M$ and $\delta:\Omega^*(M)\to \Omega^{*-1}(M)$  the Hodge codifferential.
Given a $k$-form $\omega$, we denote by $\eps_\omega$ the operator on $\Omega^*(M)$ defined by
$\eps_\omega\theta=\omega\wedge\theta$, for every $\theta\in \Omega^l(M)$.
 In Theorem~\ref{main}, we prove the following expression for the graded commutator of $\delta$ with $\eps_\omega$ in terms of Fr\"olicher-Nijenhuis operators
(to be defined later)
\begin{equation}
	\label{main-formula}
		\left[ \delta, \eps_\omega \right] = \eps_{ \delta \omega}
		 - \lie_{\omega^\#} - \left( -1 \right)^k i_{\omega^\dm}.
\end{equation}
Here, $\omega^\#\in \Omega^{k-1}(M,TM)$ denotes the vector valued form  obtained
from \mbox{$\omega\in \Omega^k(M)$} by metric contraction on the last coordinate and $\omega^\dm\in \Omega^{k}(M,TM)$ is a vector valued $k$-form defined in Section~\ref{sectionHC}.

Let $\xi$ be a vector field and $\eta$ its metric dual $1$-form.  In
Corollary~\ref{cor:goldberg} we show that in this case Formula~\eqref{main-formula} takes the form
\begin{equation}
	\label{goldberg-introduction}	
\{\delta, \epsilon_\eta\} + \lie_\xi = \epsilon_{\delta \eta} +
	i_{(\lie_\xi g)^\#},
\end{equation}
where the curly bracket denotes the anticommutator.
Equation~\eqref{goldberg-introduction} was stated by Goldberg
in~\cite{goldberg-paper} and on page~109 of~\cite{goldberg-book}.
In both cases, Goldberg refrained from explicitly proving this result. Nevertheless, he
proved a partial case of~\eqref{goldberg-introduction} on pages 110-111
of~\cite{goldberg-book} under the condition that $\xi$ generates a flow of   conformal
transformations.
The absence of a published proof can be one of the causes that
Equation~\eqref{goldberg-introduction} is not widely known.

Let us give a simple example of use of~\eqref{main-formula}.
Let $(M,g,J)$ be a K\"ahler manifold and let $\Omega(X,Y)=g(X,JY)$ be its
fundamental 2-form. Then $\Omega^\# =J$ is parallel and $\Omega$ is closed and coclosed. One gets easily that the associated vector valued 2-form $\Omega^\dm$ vanishes (see equation~\eqref{diamond2}).   Thus
\eqref{main-formula} becomes
\begin{equation}\label{kahler}
	[\delta, \eps_\Omega] + \lie_J = 0.
\end{equation}
Upon complexification of $\Omega^*(M)$, we can write $d = \partial +
\bar\partial$ with
\begin{equation*}
\partial \colon \Omega^{p,q}(M) \to \Omega^{p+1,q}
(M),\quad \bar\partial\colon \Omega^{p,q}(M)\to \Omega^{p,q+1}(M).
\end{equation*}
 Since
$i_J \beta = (p-q)\mathbf{i} \beta$ for all $\beta\in \Omega^{p,q}(M)$, we get that
\begin{equation*}\begin{aligned}
\lie_J \beta= [i_J,d] \beta =  [i_J, \partial + \bar\partial] \beta  = -\mathbf{i} (\bar \partial - \partial)\beta.
\end{aligned}
\end{equation*}
Thus
\begin{equation}\label{known}
[\delta, \eps_\Omega] - d^c =0
\end{equation}
where $d^c = \mathbf{i}(\bar\partial - \partial)$.
This is of course a well-known formula in K\"ahler geometry, but usually it takes several pages of
local computations to prove it.

In Theorem~\ref{quasi-iso} we show the  importance of the condition
\begin{equation}
	\label{zero1}
	[\delta, \eps_\omega] + \lie_{\omega^\#} = 0
\end{equation}
for a $p$-form $\omega$. Namely, we prove that if \eqref{zero1} holds for all
$\omega\in S$, where $S$ is a subset of the de Rham algebra $\Omega^*(M)$ of a Riemannian manifold $(M,g)$, then the subalgebra
\begin{equation*}
	\Omega^*_{\lie_{S^\#}} (M) := \left\{\, \beta \,\middle|\,
	\lie_{\omega^\#} \beta =0,\ \forall\omega\in S
	\right\}
\end{equation*}
of $\Omega^*(M)$ is quasi-isomorphic to  $\Omega^*(M)$ as  a commutative
differential graded algebra
(CDGA), with the
quasi-isomorphism given by the embedding.
Then
the cohomology ring
of $\Omega^*_{\lie_{S^\#}}(M)$ is isomorphic to the de Rham cohomology ring of $M$.
Note that in the case $M$ is K\"ahler manifold, the above mentioned quasi-isomorphism is the first step in the proof of formality of
K\"ahler manifolds given in~\cite{formality}.

Employing our formula, in Theorem~\ref{paralelity} we give a complete characterization
of  all forms $\omega$ that satisfy the condition~\eqref{zero1}. Namely, we prove that a $p$-form $\omega$ on a Riemannian manifold $(M,g)$ satisfies \eqref{zero1} if and only if one of the following cases holds:
\begin{enumerate}[$(i)$]
	\item  $p=1$ and $\omega^\#$ is a Killing vector field;
	\item $p\ge 2$ and $\omega$ is parallel.
\end{enumerate}

In Section~\ref{s:lcK} we consider the case of locally conformal K\"ahler manifolds.
By applying Formula~\eqref{main-formula}, we get the following result which in a sense generalizes Equation~\eqref{kahler}. Let $(M,J,g)$ be a locally conformal K\"ahler manifolds with fundamental $2$-form $\Omega$, Lee $1$-form $\theta$, and anti-Lee $1$-form $\eta$. Then, for any $p$-form $\beta$ we have
\begin{equation}
	\label{eqin:lck}
	[\delta,\eps_\Omega]\beta = (p- n) \eta\wedge\beta - \lie_J\beta +
	\Omega\wedge i_{\theta^\#}\beta.
\end{equation}

Finally, in Section~\ref{qSm} we show how our formula works in the context of quasi-Sasakian manifolds.
In Theorem~\ref{thmqS} we prove the following result.
	Let $(M, \phi,\xi,\eta,g)$ be a quasi-Sasakian manifold and let $A := -\phi \circ \cov \xi$. Then
	\begin{equation}
		\label{qS2i}
			\left[ \delta, \eps_\Phi \right]  = - \Tr(A) \eps_\eta - \lie_\phi + 2 \eps_\eta i_A.
	\end{equation}
The special case of  Formula \eqref{qS2i} for Sasakian manifolds was first proved by Fujitani in~\cite{fujitani} by complicated computation in local coordinates. This formula was crucial
	for the proof of the main result in our recent article \cite{jdg} on Hard Lefschetz
	Theorem for Sasakian manifolds.
	We hope that \eqref{qS2i} will allow us to obtain a suitable
	generalization of Hard Lefschetz Theorem
for quasi-Sasakian manifold.

\section{Preliminaries}
\label{sectionFN}
In this section we remind the reader of some notions and results of Fr\"olicher-Nijenhuis calculus \cite{fnstrange,fn} which will be used later.

A \emph{ commutative differential graded algebra } $\left( A,d \right)$
(CDGA for short) is a graded algebra $A = \bigoplus_{k\ge 0} A_k$ over $\mathbb{R}$ such that for all $x\in A_k$
and $y\in A_l$ we have
\[
x y = \left( -1 \right)^{kl} yx,
\]
together with a differential $d$ of degree one, such that $d(xy)=d(x)y+(-1)^k x d(y)$ and $d^2 =0$.
Let $M$ be a smooth manifold of dimension $n$.
Then the direct sum
\begin{equation*}
\Omega^*(M):= \bigoplus_{k=1}^n \Omega^k(M) \end{equation*}
is a  CDGA with the multiplication given by  the wedge product $\wedge$ and the
differential given by the exterior
derivative ${d\colon \Omega^k(M)\to \Omega^{k+1}(M)}$.

Let $(A,d)$ be a CDGA.  We say that a linear operator $D\colon A\to A$
is a
\emph{derivation of degree $p$} if $D(A_k) \subset A_{k+p}$ for all $k$, and
\begin{equation*}
D(xy) = D(x) y + (-1)^{kp} x D(y)
\end{equation*}
for all $x\in A_k$ and $y\in A_l$.

We write $\Omega^k(M,TM)$ for the space of skew-symmetric $TM$-valued
$k$-forms on~$M$.
Denote by $\Sigma_m$ the permutation group on $\left\{ 1,\dots, m \right\}$.
For $k$ and $s$ such that ${k+s = m}$, let $\shuffle_{k,s}$ be the subset of
$(k,s)$-shuffles in $\Sigma_m$.
Thus for $\sigma\in \shuffle_{k,s}$, we have
\begin{align*}
\sigma (1) < \sigma (2) <\dots <\sigma (k), && \sigma (k+1) <\dots<\sigma (k+s).
\end{align*}
Let $\phi\in \Omega^p(M,TM)$. We define the operator $i_\phi$ of degree $p-1$ on
$\Omega^*(M)$ by
\begin{align*}
	\label{iphi}
	\left( i_\phi\omega \right)\left( Y_1,\dots,Y_{p+k-1} \right) =\!\!
	\!\!\sum_{\sigma\in \shuffle_{p,k-1}} \!\!\!\!  (-1)^\sigma \omega\left( \phi(Y_{\sigma
	(1)},\dots, Y_{\sigma (p)}), Y_{\sigma (p+1)}, \dots, Y_{\sigma (p+k-1)} \right)
\end{align*}
where $\omega\in \Omega^k(M)$. The Lie derivative $\lie_\phi$ is an operator of degree $p$ on
$\Omega^*(M)$ defined as the graded commutator $\left[ i_\phi, d \right]$.

We recall now the fundamental theorem of Fr\"olicher-Nijenhuis calculus.
\begin{theorem}[{\cite{fn}}]
	\label{FN}
Let $D\colon \Omega^*(M)\to \Omega^*(M)$ be a derivation of degree
$p$. Then there
are unique $\phi\in\Omega^p(M,TM)$ and $\psi\in\Omega^{p+1}(M,TM)$, such that $D = \lie_\phi + i_\psi$.
\end{theorem}
As a consequence of the above theorem, we get:
\begin{enumerate}[(i)]
	\item If a $TM$-valued $p$-form $\phi$ is different from $0$, then
		$i_\phi\not=0$.
	\item If $D \colon \Omega^*(M)\to \Omega^*(M)$ is a
		derivation such that $\left[ D, d \right] =0$, then
		there is  a unique $\phi\in\Omega^p(M,TM)$, such that
		$D = \lie_\phi$.
\end{enumerate}

For a $k$-form $\omega\in \Omega^k (M)$ and $TM$-valued $p$-form $\phi$ , we
define the $TM$-valued
$(p+k)$-form $\omega \wdd \phi$ by
\begin{equation*}
	\left( \omega\wdd \phi \right)\left( Y_1,\dots, Y_{p+k} \right) =
	\sum_{\sigma\in\shuffle_{k,p}} (-1)^\sigma \omega\left( Y_{\sigma
	(1)},\dots,
	Y_{\sigma (k)} \right) \phi( Y_{\sigma (k+1)}, \dots, Y_{\sigma (k+p)}
).
\end{equation*}
Following~\cite{fnstrange}, we will define the \emph{contraction} (sometimes called \emph{trace}) operator
$\tr\colon \Omega^p(M,TM)\to \Omega^{p-1}(M)$ as follows. Every $\phi\in
\Omega^p(M,TM)$ can be written locally as a finite sum $\sum_{i\in I} \omega_i
\wdd X_i$, where $X_i$ are vector fields and $\omega_i\in \Omega^p(M)$. Then
\begin{equation*}
	\tr(\phi):= \sum_{i\in I} i_{X_i} \omega_i.
\end{equation*}
One can check that $\tr(\phi)$ does not depend on the choice of the
local presentation for $\phi$.
We will use the following property~\cite[eq. (2.12)]{fnstrange}
\begin{align}
	\label{contraction}
	\tr(\omega\wedge \phi) = (-1)^k \omega\wdd \tr(\phi) + (-1)^{(k+1)p}
	i_\phi \omega,
\end{align}
for any $\omega\in \Omega^k(M)$ and $\phi\in \Omega^p(M,TM)$.
	Given $\omega\in \Omega^k(M)$, we define
\begin{align*}
	\eps_\omega\colon \Omega^p(M,TM) & \to \Omega^{p+k}(M,TM)\\
\phi & \mapsto \omega \wdd \phi.
\end{align*}
For an operator $A\colon \Omega^*(M)\to \Omega^*(M)$ and $\omega \in
\Omega^*(M)$ we abbreviate the
composition $\eps_\omega \circ A$ by $\omega\wedge A$.
It is easy to check that
\begin{equation}
	\label{omegaiphi}
	\omega \wedge i_{\phi} = i_{\omega\wdd\phi}.
\end{equation}

We will need the following fact.
\begin{proposition}\label{prop:lie-wedge}
Let $M$ be a smooth manifold, $\omega\in \Omega^k(M)$, and $\phi\in\Omega^p(M,TM)$.
Then,
\begin{align}\label{lie-wedge}
\omega\wedge \lie_\phi =  \lie_{\omega\wdd \phi} -(-1)^{p+k} i_{(d\omega)\wdd
\phi}.
\end{align}
\end{proposition}
\begin{proof}
The computation
\begin{align*}
	\lie_{\omega\wdd \phi}& = \left[  i_{\omega\wdd \phi}, d \right] = \left[
	\omega\wedge i_\phi, d \right] = (-1)^{k+p}(d\omega)\wedge i_{\phi} + \omega\wedge
	\lie_\phi.
\end{align*}
proves the claim.
\end{proof}

\section{Generalized Goldberg Formula} \label{sectionHC}
In this section we prove the main result of the article.
Let  $M$ be a smooth manifold equipped with a Riemannian metric $g$ and let $\cov$ denote the corresponding
Levi-Civita connection.
Using $\cov$, we can define the map
\begin{equation*}
d^\cov\colon \Omega^p(M,TM)\to
\Omega^{p+1}(M,TM)
\end{equation*}
 similarly to the standard exterior derivative, as follows
\begin{align*}
	d^\cov \phi\left( Y_1,\dots, Y_{p+1} \right) =&
	\sum_{s=1}^{p+1}(-1)^{s-1}
	\cov_{Y_s}\left( \phi( Y_1,\dots, \widehat{Y}_s,\dots, Y_{p+1}
	)
	\right)\\& + \sum_{s<t}(-1)^{s+t}\phi\left( \left[ Y_s,Y_t
	\right],Y_1,\dots,
	\widehat{Y}_s,\dots, \widehat{Y}_t,\dots, Y_{p+1} \right).
\end{align*}
Since for the Levi-Civita connection we have
$
	[Y,Z] = \cov_Y Z - \cov_Z Y,
$
one can easily check that
\begin{equation}
	\label{dcov}
	(d^\cov \phi)(Y_1, \dots, Y_{p+1}) =  \sum_{s=1}^{p+1} (-1)^{s+1}
	(\cov_{Y_s} \phi) (Y_1,\dots, \widehat{Y}_s, \dots, Y_{p+1}).
\end{equation}
Moreover, note that $d^\cov$ is related to the Riemann curvature by the formula
\begin{equation*}
	(d^\cov)^2 \phi(Y_1, \dots, Y_{p+2}) =  \sum_{\sigma\in \shuffle_{2,p}}   (-1)^\sigma
	R(Y_{\sigma(1)},Y_{\sigma(2)}) \left(\phi (Y_{\sigma(3)},\dots, Y_{\sigma(p+2)})\right).
\end{equation*}
For $\omega\in\Omega^k(M)$ and $\phi\in \Omega^p(M,TM)$, we have
$$
d^\cov \left( \omega\wdd \phi \right) = (d\omega)\wdd \phi + (-1)^k \omega \wdd
\left( d^\cov \phi \right).
$$
Note that
for any vector field $X\in \Omega^0(M,TM)$, we get
$$
d^\cov X \left( Y \right) = \cov_Y X.
$$
Hence, $d^\cov X = \cov X$. Thus we can think about $\cov$-parallel vector fields
as a generalization of harmonic functions.
For any $k$-form $\omega$
and  any vector field $X$, we get
$$
\lie_X \omega = \cov_X \omega + i_{\cov X }\omega .
$$
In other words
\begin{equation}
	\label{covlie}
\cov_X = \lie_X - i_{d^\cov X}.
\end{equation}
This equation  suggests the following generalization of the covariant derivative. Namely, for $\phi\in \Omega^p(M,TM)$ we define
\begin{equation}
	\label{covphi}
	\cov_{\phi} := \lie_\phi -(-1)^p i_{d^\cov\phi}.
\end{equation}
	We get
\begin{align*}
	\omega\wedge \cov_\phi &= \omega\wedge \lie_\phi - \omega \wedge
	i_{d^\cov \phi} = \lie_{\omega \wdd \phi} -(-1)^{p+k}i_{(d\omega)\wdd
	\phi} - (-1)^p  i_{\omega \wdd d^\cov \phi}
	\\[2ex] &= \lie_{\omega\wdd \phi} -(-1)^{p+k} i_{d\omega \wdd \phi + (-1)^k
	\omega \wdd d^\cov \phi}= \lie_{\omega\wdd \phi} -(-1)^{p+k} i_{d^\cov(\omega \wdd \phi)}
\end{align*}
that is,
\begin{equation}
	\label{omegacovphi}
\omega\wedge \cov_\phi = \cov_{\omega \wdd\phi}.
\end{equation}
This equation is a generalization of the property
\[
f\nabla_X=\nabla_{fX}
\]
for the usual covariant derivative, where $f\in C^\infty(M)$ and $X\in \Omega^0(M,TM)$.

The Hodge codifferential is abstractly defined as the Hodge dual of the
operator $d$ on $\Omega$. It is well known that given a local orthonormal frame
$X_1$, \dots, $X_n$ on $U\subset M$, the following local expression for the codifferential holds
$$
\delta=-\sum_{t=1}^n i_{X_t}\circ \cov_{X_t}.
$$
Since both $i_{X_t}$ and $\cov_{X_t}$ are derivations of
$\Omega^*(U)$, we see that $\delta$ is a differential operator of order
$2$ on $\Omega^*(U)$, and thus also on $\Omega^*(M)$.

Let $\omega\in \Omega^p(M)$.  Then
$\left[ \delta, \eps_\omega \right]$ is a differential operator of order
$1$ and of degree $p-1$ on $\Omega^*(M)$. Thus it can be expressed in a unique
way as a sum
$$
\eps_\alpha  + \cov_\phi + i_\psi
$$
for suitable $(p-1)$-form $\alpha$, $TM$-valued $(p-1)$-form $\phi$, and
$TM$-valued $(p+1)$-form $\psi$.
Our aim is to identify $\alpha$, $\phi$ and $\psi$ for a given $\omega$.

For $\omega\in\Omega^p(M)$, we define $\omega^\#\in \Omega^{p-1}(M,TM)$ and
$\omega^\cov\in\Omega^p(M,TM)$ by
\begin{align}
	\label{omegasharp}
	\omega^\# &= \sum_{t=1}^n (i_{X_t}\omega)\wdd X_t & \omega^\cov & =
	\sum_{t=1}^n (\cov_{X_t}\omega)\wdd X_t.
\end{align}
It is easy to see that $\omega^\#$ and $\omega^\cov$ do not depend on the choice
of the orthonormal frame $X_1$, \dots, $X_n$.
Therefore $\omega^\#$ and  $\omega^\cov$ are well-defined.
By applying the contraction operator $\tr$ to \eqref{omegasharp}, we get
\begin{align}
	\label{trsharp}
	\tr(\omega^\#)& = \sum_{t=1}^n i_{X_t}^2 \omega =0\\
	\label{deltacov}
	\tr(\omega^\cov) & = \sum_{t=1}^n i_{X_t} \cov_{X_t} \omega =
	-\delta\omega.
\end{align}

\begin{proposition}
	\label{prop:omegacov}
	For any $\omega\in \Omega^p\left( M \right)$ , we have $d^\cov\left(
	\omega^\# \right) + \left( d\omega \right)^\# = \omega^\cov$.
\end{proposition}
\begin{proof}
	Let $X_1$, \dots, $X_n$ be an orthonormal frame on an open set
	$U$ in $M$.
By definition of $\omega^\cov$ and the Leibniz rule for $d^\cov$, we get
\begin{equation}
	\label{dcovomegasharp}
	d^\cov\left( \omega^\# \right) = \sum_{t=1}^n d\left( i_{X_t} \omega
	 \right)\wdd X_t + \left( -1 \right)^{p-1} \sum_{t=1}^n i_{X_t} \omega \wdd
	\cov X_t.
\end{equation}
Further,
\begin{equation}
	\label{domegasharp}
	\left( d\omega \right)^\# = \sum_{t=1}^n i_{X_t}\left( d\omega \right)
	\wdd X_t.
\end{equation}
Note that for every $1\le t\le n$, we have
\begin{equation*}
	d\left( i_{X_t} \omega\right) + i_{X_t} \left( d\omega \right) =
	\lie_{X_t}\omega = \cov_{X_t} \omega + i_{\cov X_t} \omega.
\end{equation*}
Therefore, summing \eqref{dcovomegasharp} with \eqref{domegasharp} we get
\begin{align*}
	d^\cov\left( \omega^\# \right) + \left( d\omega \right)^\# &=
	\sum_{t=1}^n \cov_{X_t} \omega \wdd X_t + \sum_{t=1}^n i_{\cov X_t}
	\omega \wdd X_t + \left( -1 \right)^{p-1} \sum_{t=1}^n i_{X_t} \omega \wdd
	\cov X_t \\& = \omega^\cov+ \sum_{t=1}^n i_{\cov X_t}
	\omega \wdd X_t + \left( -1 \right)^{p-1} \sum_{t=1}^n i_{X_t} \omega \wdd
	\cov X_t.
\end{align*}
Let us denote the expression
\begin{equation*}
 \sum_{t=1}^n i_{\cov X_t}
	\omega \wdd X_t + \left( -1 \right)^{p-1} \sum_{t=1}^n i_{X_t} \omega \wdd
	\cov X_t
\end{equation*}
by $T$. Since $T = d^\cov\left( \omega^\# \right) + \left( d\omega \right)^\# -
\omega^\cov$, we see that $T$ does not depend on the choice of the orthonormal
basis $X_1$, \dots, $X_n$ and that $T$ is a tensor on $M$. Let $x\in M$. Then there is an
local orthonormal frame $X_1$, \dots, $X_n$ on an open neighbourhood of $x$ such that $\left( \cov X_t
\right)_x = 0$
for every $1\leq t\leq n$. Computing $T_x$ with respect to this basis, we see that
$T_x = 0$. Since $x$ is an arbitrary point of  $M$, we see that~$T \equiv 0$.
\end{proof}
Let us define for every $\omega\in \Omega^p(M)$ the $TM$-valued form
\begin{equation}
	\label{diamond0}
\omega^\dm=d^\cov\left( \omega^\# \right) + \omega^\cov.
\end{equation}
Note that by Proposition~\ref{prop:omegacov} we can write it in two
other ways
\begin{align}
	\label{diamond1}
	\omega^\dm &= 2d^\cov\left( \omega^\#
	\right) + \left( d\omega \right)^\#,\\
	\label{diamond2}
	\omega^\dm &= 2\omega^\cov  - \left( d\omega
	\right)^\#.
\end{align}
Now \eqref{trsharp} and \eqref{deltacov}  give the following
expression for $\delta \omega$ in terms of~$\omega^\dm$
\begin{align}
	\label{deltaomega}
	\delta\omega= -\frac12 \tr(\omega^\dm).
\end{align}
We can now prove the announced  formula for the commutator of the codifferential with the left wedge multiplication by a $k$-form.
\begin{theorem}
	\label{main}
	Let $\omega\in \Omega^p(M)$. Then
	\begin{equation}\label{liephithm}	
		\left[ \delta, \eps_\omega \right] = \eps_{\delta \omega} -
		\cov_{\omega^\#} - (-1)^p i_{\omega^\cov}.
		\end{equation}
	Or, using the Lie derivative instead of the covariant derivative,
	\begin{equation}	
		\label{ourgoldberg}
		\left[ \delta, \eps_\omega \right] = \eps_{ \delta \omega}
		 - \lie_{\omega^\#} - \left( -1 \right)^p i_{\omega^\dm}.
		\end{equation}
\end{theorem}
\begin{proof}
Let $X$ be a vector field and $\omega\in \Omega^p\left( M \right)$. Then
\begin{align*}
	\left[ i_X\circ \cov_X, \eps_\omega \right] &= \left[ i_X,
	\eps_\omega
	\right]\circ \cov_X + i_{X}\circ \left[ \cov_X, \eps_\omega \right]
	\\&= \eps_{i_X \omega} \cov_X + i_X  \eps_{\cov_X \omega}
	\\&= \eps_{i_X \omega} \cov_X + [i_X,  \eps_{\cov_X \omega}] + (-1)^p
	\eps_{\cov_X \omega} i_X
	\\& = \cov_{i_X\omega \wdd X} + \eps_{i_X\cov_X \omega} + (-1)^p
	\eps_{\cov_X \omega} i_X
	\\&= \eps_{i_X\cov_X \omega} + \cov_{i_X \omega \wdd X} + (-1)^p
	i_{\cov_X \omega \wdd X}.
\end{align*}
Now  \eqref{liephithm} follows by substituting $X_t$ instead of $X$ and summing up over
$t$.

Since $\omega^\#\in \Omega^{p-1}\left( M,TM \right)$, from \eqref{covphi}  we
get
\begin{equation*}
	\cov_{\omega^\#} = \lie_{\omega^\#} - \left( -1 \right)^{p-1}
	i_{d^\cov\left( \omega^\# \right)} = \lie_{\omega^\#} + \left( -1 \right)^{p}
	i_{d^\cov\left( \omega^\# \right)}.
\end{equation*}
Therefore
\begin{equation*}
	\left[ \delta, \epsilon_\omega \right] = \epsilon_{\delta\omega} -
	\lie_{\omega^\#} - (-1)^p \left( i_{d^\cov\left( \omega^\# \right)} +
	i_{\omega^\cov} \right).	
\end{equation*}
\end{proof}
As a corollary we can get Formula~(4) in Goldberg's
article~\cite{goldberg-paper}.
\begin{corollary}
	\label{cor:goldberg}
Let $\xi$ be a vector field on a Riemannian manifold $M$, and $\eta$ its metric
dual $1$-form. Then $\eta^\dm = (\lie_\xi g)^\#$, that is
\begin{equation}
	\label{goldberg}
	\{\delta, \epsilon_\eta\} + \lie_\xi = \epsilon_{\delta \eta} +
	i_{(\lie_\xi g)^\#},
\end{equation}
where $\{-,-\}$ denotes the anti-commutator of operators and $(\lie_\xi g)^\#$
is the metric contraction of the $(0,2)$-tensor $\lie_\xi g$.
\end{corollary}
\begin{proof}
	We have to check that $d^\cov \eta^\# + \eta^\cov = (\lie_\xi
	g)^\#$.
	Since $\eta^\# = \xi$, we have for any vector field $Y$
	\begin{equation}
		\label{dcovxi}
		(d^\cov \eta^\#)(Y) = (d^\cov \xi)(Y) = \cov_Y \xi = \sum_{t=1}^n
		g(X_t, \cov_Y \xi)X_t,
	\end{equation}
where $X_1$, \dots, $X_n$ is a local orthonormal frame on $M$.
Further,
\begin{align}
	\label{etacov}
	\eta^\cov (Y) = \sum_{t=1}^n (\cov_{X_t}\eta)(Y)X_t = \sum_{t=1}^n
	g(\cov_{X_t} \xi, Y)X_t.
\end{align}
It is well known  that
\begin{equation}
	\label{liexi}
(\lie_\xi g)(Y,Z) = g(\cov_Y \xi, Z) + g(\xi, \cov_Z \xi),
\end{equation}
for any vector fields $\xi$, $Y$ and $Z$. Therefore, adding~\eqref{dcovxi}
and~\eqref{etacov}, we get
\begin{equation*}
	(d^\cov \xi + \eta^\cov)(Y) = \sum_{t=1}^n (\lie_\xi g)(X_t,Y )X_t =
	(\lie_\xi g)^\# (Y).
\end{equation*}
\end{proof}

Let $S$ be a set of differential forms on $M$. We will denote by $S^\#$ the set
of vector valued forms $\omega^\#$, where $\omega\in S$. Further we write
$\Omega^*_{\lie_{S^\#}}(M)$ for the intersection of the kernels of operators
$\lie_{\omega^\#}$, for all $\omega\in S$.

Recall that a morphism of CDGAs is a morphism of algebras which preserves the degree and commutes with the differentials.
Let $f:\left( A,d \right) \longrightarrow \left( B ,d \right)$ be a morphism of CDGAs. For every $k\geq 0$, the map $f$ induces a
morphism between the  $k$-th cohomologies
$$
H^k\left( f \right)\colon H^k\left( A \right) \to H^k\left( B \right).
$$
If all the morphisms $H^k\left( f \right)$ are isomorphisms then $f$ is called a \emph{quasi-isomorphism} of CDGAs.

We have the following theorem that generalizes several known facts.
\begin{theorem}
	\label{quasi-iso}
	Let $(M,g)$ be a compact Riemannian manifold. Suppose $S\subset
	\Omega^*(M)$ is such that $[\delta, \epsilon_\omega] + \lie_{\omega^\#}
	=0$ for all $\omega\in S$. Then the inclusion
\begin{equation*}
	j\colon \Omega^*_{\lie_{S^\#}}(M)
	\hookrightarrow \Omega^*(M)
\end{equation*}
is a quasi-isomorphism of CDGAs.
\end{theorem}
\begin{proof}
	Let $\omega\in S$. Since $[\delta, \eps_\omega] + \lie_{\omega^\#}
	=0$ and $\delta^2 =0$, we get that $$[\delta,\lie_{\omega^\#} ] =-[\delta, [\delta,\eps_\omega ]]=0.$$ Since the Hodge Laplacian $\lap$ is the graded commutator of $d$ and $\delta$,
	we have also that $[\lap, \lie_{\omega^\#}] =0$.

Let $\beta$ be a harmonic $p$-form. We are going to show that $\beta\in
\Omega^p_{\lie_{S^\#}}(M)$. This will imply by Hodge theory that $j$ induces a surjection in
cohomology. Since ${[\lap, \lie_{\omega^\#}] =0}$ for all $\omega\in S$, we get
immediately, that $\lap(\lie_{\omega^\#}\beta ) =0$, i.e. $\lie_{\omega^\#}
\beta$ is harmonic.
But, since $\beta$ is closed, we have $\lie_{\omega^\#} \beta = d i_{\omega^\#}
\beta$ is an exact form. Thus by Hodge theory, $\lie_{\omega^\#} \beta =0$.

It is left to show that $j$ induces an injection in cohomology. Let $\beta\in
\Omega^p_{\lie_{S^\#}}(M)$ such that $[\beta] =0$ in $H^p(M)$. Then $\beta = d
G\delta \beta$, where $G$ is the Green operator for $\lap$.
We are going to show that $G\delta \beta \in \Omega^p_{\lie_{S^\#}}(M)$. For
this, it is enough to prove that $\lie_{\omega^\#} G = G \lie_{\omega^\#}$ for
every $\omega\in S$. In fact, then
\begin{equation*}
	\lie_{\omega^\#} G \delta \beta = G\delta \lie_{\omega^\#} \beta =0, \
	\forall \omega\in S.
\end{equation*}
We have
\begin{align}
	\label{green}
	I - G\lap &= \Pi_{\lap}, & I - \lap G = \Pi_{\lap},
\end{align}
where $\Pi_\lap$ is the orthogonal projection on the set of harmonic forms.
Now we multiply the equation $\lie_{\omega^\#} \lap = \lap \lie_{\omega^\#}$ by
$G$ on the left and right hand sides. We get
\begin{equation*}
	G \lie_{\omega^\#} \lap G = G \lap \lie_{\omega^\#} G.
\end{equation*}
Applying \eqref{green} we obtain
\begin{equation}
	\label{long}
	G \lie_{\omega^\#} - G\lie_{\omega^\#}\Pi_{\lap} = \lie_{\omega^\#} G -
	\Pi_{\lap}\lie_{\omega^\#}G.
\end{equation}
As we saw above, $\lie_{\omega^\#}$ annihilates harmonic forms, hence
$\lie_{\omega^\#} \Pi_\lap=0$.
To finish the proof it is enough to check that $\Pi_\lap \lie_{\omega^\#} =0$.
Let $\alpha\in \Omega^k(M)$. By Hodge theory, we can write $\alpha$ as
$\alpha_\delta + \alpha_\lap + \alpha_d$, where $\alpha_\delta$ is in the image
of $\delta$, $\alpha_d$ is in the image of~$d$, and $\alpha_\lap$ is harmonic.
Note that $\lie_{\omega^\#} \alpha_\lap=0$. Further, $\lie_{\omega^\#}
\alpha_d = \pm di_{\omega^\#}\alpha_d$, where the sign depends on the degree of
$\omega$. In particular, $\lie_{\omega^\#} \alpha_d$ is exact, and therefore
$\Pi_\lap \lie_{\omega^\#} \alpha_d =0$. Finally, since $[\delta, \epsilon_\omega]
+ \lie_{\omega^\#} =0$, we get
\begin{equation*}
	\lie_{\omega^\#} \alpha_\delta = - [\delta, \epsilon_\omega]
	\alpha_\delta = - \delta (\omega\wedge \alpha_\delta).
\end{equation*}
Hence, $\lie_{\omega^\#} \alpha_\delta$ is a coexact form and thus $\Pi_\lap \lie_{\omega^\#} \alpha_\delta =0$.
\end{proof}
The previous theorem shows the importance of the property $[\delta, \omega] + \lie_{\omega^\#} =0$ for a differential form $\omega$.
In the following theorem we characterize all the forms with this property.
\begin{theorem}
	\label{paralelity}
Let $(M,g)$ be a Riemannian manifold and $\omega$ a $p$-form on $M$, with $p\ge 1$. Then
\begin{equation*}
	[\delta, \epsilon_\omega] + \lie_{\omega^\#} = 0
\end{equation*}
if and only if one of the following conditions holds
\begin{enumerate}[$(i)$]
	\item  $p=1$ and $\omega^\#$ is a Killing vector field;
	\item $p\ge 2$ and $\omega$ is parallel.
\end{enumerate}
\end{theorem}
\begin{proof}
	Let us consider first the case $p=1$.
	Suppose $\xi=\omega^\#$ is Killing. Then $\lie_\xi g = 0$. By
	Corollary~\ref{cor:goldberg}, we have
	\begin{equation*}
		\omega^\dm =(\lie_\xi g)^\#=0.
	\end{equation*}
		Applying~\eqref{deltaomega}, we get $\delta \omega = -\frac12
	\tr({\omega^\dm})
	=0$.
By~\eqref{goldberg}, we obtain that $\{\delta, \epsilon_\omega\} + \lie_\xi =0$.

Now, suppose that $\left\{\delta, \epsilon_\omega  \right\} + \lie_\xi =0$. Then
from~\eqref{goldberg}
\begin{equation}
	\label{ero}
	\epsilon_{\delta \omega} + i_{(\lie_\xi g)^\#} =0.
\end{equation}
Applying~\eqref{ero} to the constant function with the value $1$, we get $\delta\omega =0$. Thus
$i_{(\lie_\xi g)^\#} =0$.
By
Theorem~\ref{FN}, we have $\lie_\xi g=0$, and thus $\xi$ is a Killing vector field.

Now suppose $p\ge 2$ and $\cov \omega =0$. Then, by looking at defining formulae
one readily sees that $\delta \omega =0$, $d\omega =0$, and $\omega^\cov
=0$. Thus, by~\eqref{ourgoldberg} we get that $[\delta, \epsilon_\omega] +
\lie_{\omega^\#}=0$.

Finally, suppose that $[\delta, \eps_\omega] +\lie_{\omega^\#} =0$. Then,
by~\eqref{ourgoldberg} we have
\begin{equation}
	\label{zero}
	\epsilon_{\delta \omega} -(-1)^p i_{ \omega^\dm} =0.
\end{equation}
Applying~\eqref{zero} to the constant function $1$, we get that $\delta\omega=0$.
Therefore ${i_{\omega^\dm} =0}$ and,  by Theorem~\ref{FN},
we have $\omega^\dm=0$.
Using \eqref{diamond2} and \eqref{omegasharp}, we obtain
\begin{equation*}
	0= \omega^\dm=	\sum_{t=1}^n 2\cov_{X_t} \omega\wdd X_t -
	\sum_{t=1}^n i_{X_t} \omega\wdd X_t  = \sum_{t=1}^n (2\cov_{X_t} \omega
	- i_{X_t} d \omega)\wdd X_t,
\end{equation*}
where $X_1$, \dots, $X_n$ is a local orthonormal frame on $M$. Since
$X_1$, \dots, $X_n$ are linearly independent at every point, we obtain that
\begin{equation*}
	2\cov_{X_t}\omega = i_{X_t} d\omega
\end{equation*}
for all $t$. But this implies
\begin{equation}
	\label{covid}
2\cov_Z \omega = i_Z d\omega
\end{equation}
	for every
vector field $Z$.

 Let $Y_0$, \dots,
$Y_p$ be vector fields.
Then,  by using~\eqref{covid} we get
\begin{align*}
	2(d\omega)(Y_0,\dots, Y_p)&= \sum_{s=0}^p (-1)^s
	(2\cov_{Y_s}\omega)(Y_0,\dots, \widehat{Y_s}, \dots, Y_p)\\
	& =
	\sum_{s=0}^p (-1)^s (i_{Y_s} d\omega)(Y_0,\dots, \widehat{Y_s},\dots,
	Y_p) \\
	&= \sum_{s=0}^p (d\omega)(Y_0,\dots, Y_p) = (p+1)d\omega(Y_0,\dots, Y_p).
\end{align*}
Since $p\not=1$, we obtain $d\omega=0$. Now \eqref{covid} implies
$\cov\omega=0$.
\end{proof}

	\section{Locally conformal K\"ahler manifolds}
	\label{s:lcK}
In this section, we show how Theorem~\ref{main} works in the context of locally conformal K\"ahler manifolds.

	Let $(M^{2n+2},g)$ be a Riemannian manifold and $J$ a complex structure on $M$.
	Then $(M,J,g)$ is called \emph{Hermitian} if $g(JX,JY) = g(X,Y)$ for all
	vector fields $X$, $Y$ on $M$. For an Hermitian manifold $(M,J,g)$, we define its \emph{fundamental $2$-form}
	$\Omega$ by $\Omega(X,Y) = g(X,JY)$. Thus $\Omega^\# = J$. An Hermitian
	manifold $(M,J,g)$ is called \emph{locally conformal K\"ahler}
	(l.c.K.) if there exists a $1$-form $\theta$ (called the \emph{Lee form}) such that
	\begin{equation*}
		d \Omega = \theta \wedge \Omega.
	\end{equation*}
We are going to apply Theorem~\ref{main} to $\omega = \Omega$.
For this we have to compute $\Omega^\dm$ and $\delta \Omega$.
	We define $\eta = i_J \theta$. It is proved in
	\cite[Corollary~1.1]{booktad} that
	\begin{equation*}
		(\cov_X J)Y = \frac12 \left( \eta(Y) X - \theta (Y) JX -
		g(X,Y) \eta^\# - \Omega(X,Y)\theta^\# \right).
	\end{equation*}
Thus
\begin{align*}
	d^\cov J(X,Y) &=(\cov_X J)Y - (\cov_Y J)X \\
                    &= \frac12 \left(\eta(Y) X - \theta(Y) JX - \eta(X) Y +
	\theta(X) JY  - 2 \Omega(X,Y) \theta^\#\right)\\& =  \frac12 (- (\eta \wdd \Id) (X,Y) + (\theta \wdd
	J)(X,Y) ) -  (\Omega \wdd \theta^\#) (X,Y).
\end{align*}
Hence, we get
\begin{equation*}
	d^\cov J = \frac12 ( \theta \wdd J - \eta \wdd \Id) - \Omega\wdd
	\theta^\#.
\end{equation*}
Using the definition of $\#$, it is easy to check  that
\begin{equation}\label{domegasharp2}
	(d\Omega)^\# = (\theta \wedge \Omega)^\# = \Omega \wdd \theta^\# - \theta
	\wdd \Omega^\# = \Omega \wdd \theta^\# - \theta \wdd J.
\end{equation}
Thus by \eqref{diamond1}
\begin{equation}\label{omegadiamond}
\Omega^\dm = 2 d^\cov J + (d\Omega)^\# =  - \eta \wdd \Id -
\Omega\wdd \theta^\# .	
\end{equation}
Moreover, due to~\eqref{trsharp}, by contracting \eqref{domegasharp2}   we get
\begin{equation*}
\tr (\Omega \wdd \theta^\#) = \tr (\theta \wdd J)
\end{equation*}
Hence by~\eqref{deltaomega},  we obtain from~\eqref{omegadiamond}
\begin{equation*}
	\delta \Omega =- \frac12 \tr(\Omega^\dm)
    = \frac12 (\tr (\eta \wdd \Id) + \tr (\Omega \wdd \theta^\#))
    = \frac12 (\tr (\eta \wdd \Id) + \tr (\theta \wdd J)) .
\end{equation*}
Using~\eqref{contraction}, we have
\begin{align*}
	\tr(\eta\wdd \Id) & = - \tr(\Id) \eta + i_\Id \eta = - (2n+2)\eta + \eta
	= -(2n+1)\eta,\\
	\tr(\theta \wdd J) &= - \tr(J)\theta + i_J \theta = \eta.
\end{align*}
Therefore
\begin{equation*}
	\delta \Omega =\frac12 ( \eta - (2n+1) \eta)  = - n \eta.
\end{equation*}
Applying Theorem~\ref{main}, we get the following formula that in a sense generalizes Equation~\eqref{kahler}
 which holds for K\"ahler manifolds.
\begin{theorem}
	\label{thmlck}
Let $(M,J,g)$ be a locally conformal K\"ahler manifold. Let $\Omega$ be the fundamental $2$-form, $\theta$ the Lee $1$-form, and $\eta = i_J \theta$. Then, for any $p$-form $\beta$ we have
\begin{equation}
	\label{eq:lck}
	[\delta,\eps_\Omega]\beta = (p- n) \eta\wedge\beta - \lie_J\beta +
	\Omega\wedge i_{\theta^\#}\beta.
\end{equation}
\end{theorem}

\section{Quasi-Sasakian manifolds}\label{qSm}

In this section we will show how Theorem~\ref{main} can be used to get
useful formulae for commutators on quasi-Sasakian manifolds.

Recall that  an \emph{almost contact metric} structure on a manifold $M^{2n+1}$ is a
quadruple $(\phi, \xi, \eta, g)$, where $\phi$ is an endomorphism of $TM$,
$\xi$ is a vector field, $\eta$ is a $1$-form, and $g$ is a Riemannian metric
such that
\begin{align*}
	\phi^2 & = -\Id + \eta \otimes \xi, & \eta(\xi) & = 1,\\
	g(\phi X, Y) & = - g(X,\phi Y) , & \eta(X) & = g(X,\xi),
\end{align*}
for any vector fields $X$ and $Y$. As a consequence, one easily gets that $\phi(\xi) = 0$ and $\eta \circ \phi= 0$.
We define an almost complex structure $J$ on $M\times \R$ by
\begin{equation*}
	J\Big(X, f \frac{d}{dt}\Big) = \Big(\phi X - f \xi,  \eta(X)
	\frac{d}{dt} \Big),
\end{equation*}
where $f$ is a smooth function on $M\times \R$.
If $J$ is integrable, the almost contact metric structure $(\phi, \xi, \eta, g)$
on $M$ is called \emph{normal}. We define a $2$-form $\Phi$ by
\begin{equation*}
	\Phi(X,Y) = g (X, \phi Y), \ \mbox{for any } X,Y\in \vf(M).
\end{equation*}
A normal almost contact metric structure $(\phi, \xi, \eta, g)$ on $M$ is called
\emph{quasi-Sasakian} if $\Phi$ is closed.

Let $(M^{2n+1}, \phi, \xi,\eta,g)$ be
a quasi-Sasakian manifold. We define
\begin{equation*}
	A := -\phi \circ \cov \xi.
\end{equation*}
We are going to apply Theorem~\ref{main} to $\omega = \Phi$.
For this we have to compute $\Phi^\#$, $\Phi^\dm$, and $\delta \Phi$.
From the definition of $\Phi$, we have that $\Phi^\# = \phi$.
Since $\Phi$ is closed, from \eqref{diamond1}, we get
\begin{equation*}
	\Phi^\dm = 2 d^\cov \phi.
\end{equation*}
In \cite{kanemaki} it was shown that
\begin{align*}
	(\cov_X \phi) Y & = \eta(Y) AX - g(AX, Y) \xi, & g(AX,Y) & = g(X,AY).
\end{align*}
Thus by \eqref{dcov}, we have
\begin{align*}
(d^\cov \phi) (X,Y) &= (\cov_X \phi)(Y) - (\cov_Y \phi)(X)\\& =
\eta (Y) AX - g(AX,Y) \xi - \eta(X)AY + g(X,AY) \xi\\& = - (\eta \wdd A)(X,Y).
\end{align*}
Therefore
\begin{equation}
	\label{Phidm}
	\Phi^\dm = - 2 \eta \wdd A.
\end{equation}
Further by~\eqref{deltaomega}
\begin{equation}
	\label{deltaPhi}
	\delta \Phi = -\frac12 \tr (\Phi^\dm) = \tr (\eta \wdd A).
\end{equation}
By~\eqref{contraction}, we have
\begin{equation}
	\label{treta}
\tr(\eta\wdd A) = - \eta\wdd \tr(A) + i_A\eta = - \tr(A)\eta + i_A\eta.
\end{equation}
Since $A = - \phi\circ \cov \xi$ and $\eta \circ \phi =0$, combining
\eqref{deltaPhi} and \eqref{treta}, we finally get
\begin{equation*}
	\label{deltaPhi2}
	\delta \Phi = - \tr(A) \eta.
\end{equation*}
Thus by Theorem~\ref{main} and \eqref{Phidm}, we have
\begin{equation*}
		\left[ \delta, \eps_\Phi \right] = -\eps_{\tr(A) \eta} -
		\lie_{\phi} + i_{2 \eta \wdd A}.
	\end{equation*}
	Since $A$ is an endomorphism of $TM$, we actually have $\tr(A) = \Tr(A)$.
Hence we have proved the following result.
\begin{theorem}
	\label{thmqS}
	Let $(M, \phi,\xi,\eta,g)$ be a quasi-Sasakian manifold. Then
	\begin{equation}
		\label{qS2}
			\left[ \delta, \eps_\Phi \right]  = - \Tr(A) \eps_\eta - \lie_\phi + 2 \eps_\eta i_A.
	\end{equation}
\end{theorem}
The most important examples
of quasi-Sasakian manifolds are co-K\"ahler manifolds (see \cite{survey}) and
Sasakian manifolds (see \cite{boyer}).
For every co-K\"ahler manifold, one has $\cov \xi =0$ and thus $A=0$. Therefore
in co-K\"ahler case, we get
\begin{equation*}
	[\delta ,\eps_\Phi] = - \lie_\phi,
\end{equation*}
which could also have been achieved by using the fact that $\phi$
is parallel on a co-K\"ahler manifold and Theorem~\ref{paralelity}.

For Sasakian manifold, one has $\cov \xi = - \phi$, and thus $A = \phi^2 = -\Id
+ \eta \wdd \xi$.  Therefore $\Tr A = -2n$ in this case.
Applying Theorem~\ref{thmqS}, we get
\begin{equation}
	\label{S}
	\begin{aligned}
		\left[ \delta, \eps_\phi \right] & = 2n \eps_\eta -\lie_\Phi +
		2\eps_\eta
		(-i_\Id + \eps_\eta i_\xi) \\&= 2n \eps_\eta -\lie_\phi -
		2\eps_\eta
		i_{\Id}.
	\end{aligned}
	\end{equation}
	The formula \eqref{S} was first proved by Fujitani in~\cite{fujitani} by
	complicated computation in local coordinates. This formula was crucial
	for some proofs in our recent article \cite{jdg} on Hard Lefschetz
	Theorem for Sasakian manifolds.
	We hope that Theorem~\ref{thmqS} will permit us to find a suitable
	generalization of Hard Lefschetz Theorem for quasi-Sasakian manifolds.

\end{document}